\documentclass[10pt]{amsart}
\usepackage{amsfonts}
\usepackage{array}
\usepackage{tabularx}
\usepackage{arydshln}
\usepackage{amsmath}

\usepackage{amssymb}
\usepackage{physics}
\usepackage{graphicx}
\usepackage{caption}
\usepackage{subcaption}
\usepackage{multirow}
\usepackage{footmisc}
\providecommand{\U}[1]{\protect\rule{.1in}{.1in}}

\usepackage{xcolor}
\usepackage{latexsym}
\usepackage{amsmath}
\usepackage{amssymb}
\usepackage{mathrsfs}
\usepackage{graphicx}
\usepackage{color}
\usepackage{pgfpages}
\usepackage{ifthen}
\usepackage{leftidx,tensor}
\usepackage[T1]{fontenc}
\usepackage[latin1]{inputenc}
\usepackage{mathtools}
\usepackage{comment}
\usepackage{dsfont}
\usepackage[nocompress]{cite}
\usepackage[shortlabels]{enumitem}
\usepackage{aliascnt}
\usepackage[bookmarks=true,pdfstartview=FitH, pdfborder={0 0 0}, colorlinks=true,citecolor=red, linkcolor=blue]{hyperref}
\usepackage{nicefrac}
\newtheorem{thm}{Theorem}[section]
\newaliascnt{cor}{thm}
\newaliascnt{prop}{thm}
\newaliascnt{lem}{thm}
\newtheorem{cor}[cor]{Corollary}
\newtheorem{prop}[prop]{Proposition}
\newtheorem{lem}[lem]{Lemma}
\aliascntresetthe{cor}
\aliascntresetthe{prop}
\aliascntresetthe{lem}

\newaliascnt{defn}{thm}
\newaliascnt{asu}{thm}
\newaliascnt{con}{thm}

\aliascntresetthe{defn}
\aliascntresetthe{asu}
\aliascntresetthe{con}

\newcounter{stp}
\newcounter{stpi}
\newcounter{stpci}
\newcounter{stpiii}

\newaliascnt{rem}{thm}
\newaliascnt{exa}{thm}
\newaliascnt{masu}{thm}
\newaliascnt{nota}{thm}
\newaliascnt{sett}{thm}
\newtheorem{rem}[rem]{Remark}

\aliascntresetthe{rem}
\aliascntresetthe{exa}
\aliascntresetthe{masu}
\aliascntresetthe{nota}
\aliascntresetthe{sett}
\numberwithin{equation}{section}

\setlist[enumerate]{font = \normalfont}

\newcommand {\N}	{\mathbb{N}}

\newcommand {\R}	{\mathbb{R}}

\newcommand {\T}	{\mathbb{T}}

\renewcommand{\d}{\, \mathrm{d}}

\DeclareMathOperator{\divH}{div_{\H}}

\newcommand{\D}{\mathrm{D}}

\renewcommand{\H}{\mathrm{H}}

\newcommand{\sigmabar}{\bar{\sigma}}

	\newcommand{\dk}[1]{\partial_{#1}}
	\newcommand{\dt}{\dk{t}} 
	\newcommand{\dz}{\dk{z}} 
	
	\newcommand{\eps}{\varepsilon}
	\renewcommand{\phi}{\varphi}
	
	\renewcommand{\bar}[1]{\overline{#1}}
	\newcommand{\vbar}{\bar{v}}

	\renewcommand{\div}{\mathrm{div} \, }
	\newcommand{\nablaH}{\nabla_{\H}}
	\newcommand{\DeltaH}{\Delta_{\H}}

	\newcommand{\rC}{\mathrm{C}}
	\newcommand{\rL}{\mathrm{L}}
	\newcommand{\rW}{\mathrm{W}}
	\newcommand{\rH}{\H}

       \newcommand{\cV}{\mathcal{V}}
         \newcommand{\cH}{\mathcal{H}}
     \newcommand{\rI}{\mathrm{I}}

\title[Continuous Data Assimilation for Semilinear Equations: A General  Approach by Evolution Equations]{Continuous Data Assimilation for Semilinear Parabolic Equations: A General Approach by
Evolution Equations}

\author{Gianmarco Del Sarto}
\address{Technische Universit\"{a}t Darmstadt\\
Fachbereich Mathematik\\
	Schlossgartenstr.\ 7\\
	64289 Darmstadt\\
	Germany}
\email{delsarto@mathematik.tu-darmstadt.de}

\author{Matthias Hieber}
\address{Technische Universit\"at Darmstadt\\
	Fachbereich Mathematik\\
	Schlossgartenstr.\ 7\\
	64289 Darmstadt\\
	Germany}
\email{hieber@mathematik.tu-darmstadt.de}

\author{Filippo Palma}
\address{Universit\`a degli Studi della Campania L. Vanvitelli\\
	Dipartimento di Matematica e Fisica\\
	Via Vivaldi 43\\
	81100 Caserta\\
	Italy}
\email{filippo.palma@unicampania.it}

\author{Tarek Z\"{o}chling}
\address{Technische Universit\"at Darmstadt\\
	Fachbereich Mathematik\\
	Schlossgartenstr.\ 7\\
	64289 Darmstadt\\
	Germany}
\email{zoechling@mathematik.tu-darmstadt.de}

\begin{document}

\begin{abstract}
This article develops a general framework for continuous deterministic data assimilation for semilinear parabolic equations by means of evolution equations. Introducing a nudged model driven by
partial observations, the global well-posedness of the reference and the approximating systems is established under natural assumptions. In addition, it is shown that the approximating solution converges exponentially to the solution of the
reference system, provided the observational resolution and the nudging parameter are suitably chosen.
The approach allows us to consider many systems, such as the Allen-Cahn, Cahn-Hilliard, Sellers-type energy balance, and bidomain systems, for the first time. 
\end{abstract}

\maketitle
\section{Introduction}
The approach to semilinear parabolic equations via the theory of evolution equations has a long and rich
history and tradition. In many cases, one obtains via this method well-posedness results for these equations,
locally or globally, within the strong, weak, mild or variational setting, see e.\ g. \cite{Amann,Lunardi1995AnalyticSemigroups,MR3524106,GigaGigaSaal2010NPDE, Sohr2001NavierStokes,Veraar3,HieberRobinsonShibata2020NavierStokes}. In all of these approaches, a precise knowledge of the initial data and its regularity is a fundamental
ingredient in the analysis of these problems. We also refer to the setting of critical spaces for parabolic evolution equations, see e.\ g. \cite{BahouriCheminDanchin2011FourierAnalysis,PrussSimonettWilke2018CriticalSpaces, PrussWilke2018CriticalNSE}.

In many practical situations, however, the initial state of a system is only partially known or entirely unavailable. In these cases, one aims to reconstruct or approximate the solution
from partial observational data. This leads to the framework of \emph{data assimilation}, which has become an important research area in recent years, both within more theoretical and more 
applied communities \cite{WangSIAM, ReichCotter2015PFDA, Kalnay2003AtmosphericDA}. One often distinguishes between discrete or continuous-in-time data assimilation, with or without noise. We refer here to the monographs \cite{Stuart2015,Evensen}, and to the articles 
\cite{Hayden,MR4409797,Larios,Bessaih2015,TitiFoias}.

A seminal contribution to continuous data assimilation was made by Azouani, Olson, and Titi~\cite{MR3183055}, who proposed an algorithm for 
deterministic data assimilation for the two-dimensional Navier-Stokes equations. They introduced a nudged system which, like the reference system, admits a unique, global, strong
solution, and proved that the nudged solution converges exponentially in time to the reference solution in both the $\rL^2$- and the $\rH^1$-norms.

Many other models arising in mathematical physics have also been analysed within this framework, both in the context of strong and weak solutions, see e.g. \cite{FarhatJollyTiti2015Benard,MR4344886,Titialpha,Titiporus}. Most existing
analyses rely on the Faedo-Galerkin method to establish the existence of solutions to the nudged system, followed by convergence proofs based on \emph{a priori} estimates and Gronwall-type inequalities.

This paper aims to develop a general framework for continuous deterministic data assimilation for semilinear parabolic problems using the theory of evolution equations. Our approach
encompasses both an existence theory for the nudged system and a convergence result in a prescribed norm.

Let us begin by considering a semilinear evolution equation of the form
\begin{equation} \label{eq:ASE}
\left \lbrace \begin{aligned}
     u' + Au &= F(u),\\
     u(0) &= u_0,
\end{aligned}
\right.
 \tag{ASE}
\end{equation}
where $A$ is the generator of an analytic semigroup on a Banach space $X$ and  $F$ is a given nonlinear mapping, however, the initial data is not known. In this scenario, uniquely determining the trajectory of the system is impossible.

The objective is now to construct a solution to an associated problem perturbed by partial observations (the nudged system), with known initial data, and
to show that the solution of this associated equation converges exponentially, for large times, to the solution of the original problem. More precisely, let $\rI_\delta(u(t))$ be
available measurements of the state. We then study the corresponding nudged system
\begin{equation}\label{eq:model dataintro}
    \left\{
    \begin{aligned}
        v'+Av&= F(v) -\mu  ( \rI_\delta  v -\rI_\delta \tilde{u}) ,\\
        v(0)&=v_0,
    \end{aligned}
    \right.
\end{equation}
where $\tilde{u}$ denotes a suitably shifted solution of \eqref{eq:ASE}, as clarified in detail in the next section.

To establish long-term predictability of $u$, we first verify that \eqref{eq:ASE}, possibly up to a positive time shift, is globally well-posed for an appropriate class of initial data. We then
formulate general assumptions on the operator $A$ and the nonlinearity $F$ ensuring global solvability of \eqref{eq:ASE} and prove that, for suitable choices of the parameters $\delta$ and $\mu$,
the nudged system is globally well-posed and its solutions converge exponentially in time to those of \eqref{eq:ASE}, or of a positively shifted variant. Note that we do not assume the
existence of a unique, global, strong solution to \eqref{eq:ASE}; rather, we show that our assumptions (A1), (A2) and (A3) below are sufficient to guarantee such a solution.

We will illustrate our abstract approach with several applications from fluid mechanics and biomedical modelling. In the context of strong solutions, we consider the two-dimensional
Navier-Stokes equations, the three-dimensional primitive equations, a Sellers-type energy balance model \cite{Sellers1969,North17}, and the two-dimensional bidomain model
\cite{Keener,ColliFranzone}. For weak solutions, we treat the two-dimensional Navier-Stokes equations, the one-dimensional Allen-Cahn equation \cite{AllenCahn}, and the one- and two-dimensional
Cahn-Hilliard equations \cite{CahnHilliard}.

We note that our approach allows us to consider the continuous data assimilation problem in the strong setting for both the energy balance model (\autoref{subsec:EBM}) and the two-dimensional bidomain model
(\autoref{subsec:bidomain}) for the first time. In the weak setting, we obtain by our method new results for the one-dimensional Allen-Cahn equation and for the one- and two-dimensional
Cahn-Hilliard equations.

The paper is organised as follows. In \autoref{sec:prelim and main}, we introduce the general framework and develop an existence theory for the initial boundary value problem associated with \eqref{eq:ASE}. We then establish general assumptions ensuring global-in-time solvability of \eqref{eq:ASE}, up to a positive shift, and prove our main results concerning the solvability and convergence of the data assimilation system. In \autoref{sec: Example}, we apply the theory in the strong setting to models motivated by climate science and biomedical applications. Finally, in \autoref{sec: Example-weak}, we present some examples in the weak setting.

\section{Preliminaries and Main results} \label{sec:prelim and main}
\noindent
Let $(\cV,\cH,\cV^\ast)$ be a \emph{Gelfand triple} of real Hilbert spaces, meaning the embeddings 
$$
\cV \hookrightarrow \cH \hookrightarrow \cV^\ast
$$
are dense and continuous, and the pairing between $\cV$ and $\cV^\ast$ satisfies $$
\langle u,v \rangle_{\cV,\cV^\ast} = (u,v)_\cH , \qquad \forall \, u \in \cV, \ v \in \cH.
$$ 
We assume that for the real interpolation space it holds $(\cV^\ast, \cV)_{\frac{1}{2},2}= \cH$, where $(\cdot, \cdot)_{\theta,p}$ denotes the real interpolation functor for $\theta \in (0,1)$ and $p\in (1,\infty)$. We further denote the complex interpolation space $[\cV^\ast,\cV]_\beta$ by $\cV_\beta$ for $\beta \in (0,1)$. Moreover, for $m \in \mathbb{N}$ and $q \in (1,\infty)$, we denote by  $\rL^q(\Omega)$ and $\rH^{m,q}(\Omega) = \rW^{m,q}(\Omega)$ respectively the Lebesgue and Sobolev spaces and we refer to their norm as $\|\cdot\|_q$ and $\|\cdot\|_{m,q}$. For more information on function spaces we refer e.g. to \cite{AdamsFournier,Lunardi2018}.

Given a bounded operator $A \colon \cV \to \cV^\ast$ and $u_0 \in \cH$, consider the following semi-linear parabolic evolution equation
\begin{equation}\label{eq:model}
    \left\{
    \begin{aligned}
        u'+Au&=F(u) , \quad t \in (0,T),\\
        u(0)&=u_0.
    \end{aligned}
    \right.
\end{equation}
We impose the following conditions:

    \begin{description}
        \item[(A1)] 
        $A$ is quasi-coercive, i.\ e., for all $u\in \cV$ it holds that
        \begin{equation*}
            \langle Au,u \rangle_{\cV^\ast,\cV} \geq \alpha \|  u \|^2_\cV - \omega \| u \|^2_\cH \ \text{ for some } \ \alpha>0 \text{ and }\omega \geq 0.
        \end{equation*}
        \item[(A2)]
        $F \colon \cV_\beta \to \cV^\ast$ satisfies for all $u_1$, $u_2 \in \cV_\beta$ the estimate
        \begin{equation*}
            \| F(u_1) - F(u_2) \|_{\cV^\ast} \leq C \sum_{j=1}^k \bigl ( 1+ \| u_1 \|^{\rho_j}_{\cV_{\beta}} + \| u_2 \|^{\rho_j}_{\cV_{\beta}} \bigr ) \| u_1 -u_2 \|_{\cV_{\beta_j}}
        \end{equation*}
        for a constant $C>0$ and numbers $k \in \N$, $\rho_j \geq 0$, $\beta \in (\frac{1}{2},1)$ and $\beta_j \in (\frac{1}{2}, \beta ]$.
        \item[(A3)]
        For all $j=1,\dots,k$ assume
        \begin{equation*}
            \rho_j \bigl (\beta_j -\frac{1}{2} \bigr ) + \beta \leq 1.
        \end{equation*}
    \end{description}

\noindent

\begin{rem}\label{rem: bilinear}{\rm 
        Many fluid dynamics models, such as the Navier-Stokes equations, are given by the abstract formulation 
        \begin{equation*}\label{eq:model fluid}
    \left\{
    \begin{aligned}
        u'+Au&=\Phi(u,u) , \quad t \in (0,T), \\
        u(0)&=u_0,
    \end{aligned}
    \right.
\end{equation*}
where $\Phi \colon \cV_\beta \times \cV_\beta \to \cV^\ast$ is bilinear and bounded, thus $\|\Phi(u,u)\|_{\cV^\ast} \leq C \| u \|^2_{\cV_\beta}$. If $\beta \leq \nicefrac{3}{4}$, such bilinear function $\Phi$ naturally satisfies $\bf(A2)$ and $\bf(A3)$ with $j=1$, $\rho=\rho_1=1$. 
    }
\end{rem}
\begin{rem}
   {\rm  Note that, from $\bf(A1)$, it follows that $-A$ generates an analytic semigroup on $\cV^\ast$, which by general theory admits maximal $\rL^2$-regularity.}
\end{rem}
\noindent
Based on the above assumptions, we recall the following result on the local well-posedness of the abstract evolution equation \eqref{eq:model}, see \cite[Theorem 18.2.6]{Veraar3}.
\begin{lem}\label{lem: local}
    Let $T>0$ and assume that $\bf(A1)-\bf(A3)$ are satisfied. Then for any $u_0 \in \cH$, there exists $a = a(u_0)\leq T$ such that problem \eqref{eq:model} admits a unique solution
    \begin{equation}
        u \in \rL^2(0,a;\cV) \cap \rH^1(0,a;\cV^\ast) \cap \mathrm{BUC}([0,a];\cH).
    \end{equation}
    The solution exists on a maximal time interval $[0,a_{\mathrm{max}}(u_0))$ and depends continuously on the data. If the solution does not exist globally in time, i.\ e., if $a_{\mathrm{max}} <T$, the maximal existence time is characterized by a blow-up 
    \begin{equation}
        \label{eq: blow up}
          \lim\limits_{t \to a_{\mathrm{max}}} \| u \|_{ \rL^2(0,t;\cV) \cap \rH^1(0,t;\cV^\ast) } = \infty.
    \end{equation}
\end{lem}
\noindent
Now let $T>0$, $u_0 \in \cH$ be given and suppose that $\bf(A1)-\bf(A3)$ hold. We denote by $u$ the corresponding maximal solution provided in \autoref{lem: local}. To guarantee global-in-time existence, we impose the following crucial assumption.
    \begin{description}
        \item[(A4)] Let $t_+>0$ and $u \in \rL^2(0,t;\cV) \cap \rH^1(0,t;\cV^\ast)$ for all $t < t_+$. Assume that $ \lim\limits_{t \to t_+} \| F(u) \|_{ \rL^2(0,t;\cV^\ast)} < \infty$ and $\| u (t) \|^2_{\cH}$ is integrable on $(0,\infty)$. 
    \end{description}
\noindent
Denoting by $\omega \geq 0$ the constant specified in $\bf(A1)$, the latter assumption guarantees that the shifted problem 
\begin{equation}\label{eq:model shift}
    \left\{
    \begin{aligned}
        \tilde{u}'+(A+\omega)\tilde{u}&=F(\tilde{u}) , \quad t \in (0,T),\\
        \tilde{u}(0)&=u_0.
    \end{aligned}
    \right.
\end{equation}
admits a unique, global strong solution. 
\begin{cor}[Global well-posedness of \eqref{eq:model shift}]\label{cor: global shift} \mbox{} \\
    Assume that $\bf(A1)-\bf(A3)$ hold. Let $u_0 \in \cH$ and assume that the solution $\tilde{u}$ of \eqref{eq:model shift} satisfies $\bf(A4)$. Then there is a unique, global solution $\tilde{u}$ of \eqref{eq:model shift} satisfying
    \begin{equation*}
        \tilde{u} \in \rL^2(0,\infty;\cV) \cap \rH^1(0,\infty;\cV^\ast).
    \end{equation*}
\end{cor}
\begin{proof}
    By \autoref{lem: local} there is a unique maximal solution $\tilde{u}$ of \eqref{eq:model shift}. In particular, \autoref{lem: local} and $\bf(A4)$ guarantee
    \begin{equation*}
       \lim\limits_{t \to a_{\mathrm{max}}} \| \tilde{u} \|_{\rL^2(0,t;\cV) \cap \rH^1(0,t;\cV^\ast)} \leq  C \lim\limits_{t \to a_{\mathrm{max}}} \bigl ( \| F(\tilde{u}) \|_{\rL^2(0,t;\cV^\ast)} + \| u_0 \|_\cH \bigr ) < \infty
    \end{equation*}
    and therefore $\tilde{u}\in \rL^2_{\mathrm{loc}}(0,\infty;\cV) \cap \rH^1_{\mathrm{loc}}(0,\infty;\cV^\ast) \cap \mathrm{BUC}((0,\infty);\cH)$. We note that the constant $C>0$ is independent of time since the operator $A+ \omega I$ is invertible, thanks to $\bf(A1)$. To extend $\tilde{u}$ to $T=\infty$ we note that the spectrum of the shifted operator \( A + \omega \) lies entirely in the right half of the complex plane, that is, \( \sigma(A + \omega) \subset \mathbb{C}_+ \).  Specifically, since the semigroup generated by $-A- \omega$ is analytic, it implies the exponential decay of the semigroup. This in turn implies that, given $\| \tilde u (t_0) \|_\cH$ sufficiently small for some $t_0 >0,$ then the solution extends to $T=\infty$. By $\bf(A4)$ and the above considerations, $\| \tilde{u} \|_\cH^2$ is integrable and continuous on $(0,\infty)$ and it follows that $\inf_{ t \in (0,\infty)} \| \tilde{u}(t) \|^2_\cH =0$. Hence, there is $t_0 \geq 0$ such that $\| \tilde{u}(t_0) \|^2_\cH$ is sufficiently small, which concludes the proof.
\end{proof}
\noindent
We are now in a position to introduce an approximate system corresponding to this solution \( \tilde{u} \), the so-called \emph{Data Assimilation system}. Let $\tilde{u}$ denote the unique, global solution of \eqref{eq:model shift}. Consider the system
\begin{equation}\label{eq:model data}\tag{\textcolor{blue}{DA}}
    \left\{
    \begin{aligned}
        v'+Av&= F(v) -\mu  ( \rI_\delta  v -\rI_\delta \tilde{u}) , \quad t \in (0,T),\\
        v(0)&=v_0,
    \end{aligned}
    \right.
\end{equation}
where $\mu>0$ denotes a constant referred to as \emph{nudging parameter} and $\rI_\delta$ denotes a linear operator satisfying 
\begin{equation} \label{eq: I_delta bound2}
    \langle  f-\rI_\delta f, g \rangle_{\cV^\ast, \cV} \leq C \delta \| f \|_\cH \| g \|_\cV  \ \text{ for all }\ f \in \cH \ \text{ and }\ g \in \cV.
\end{equation}
The operator $\rI_\delta$ models observational measurements of the system \eqref{eq:model} at a coarse spatial resolution with a scale $\delta$. From the bound \eqref{eq: I_delta bound2} on the observation operator $\rI_\delta$ we conclude the following result:
\begin{prop}[Global well-posedness of the Data Assimilation problem \eqref{eq:model data}]\label{prop: global DA} \mbox{} \\
Suppose that the observation operator $\rI_\delta$ satisfies \eqref{eq: I_delta bound2} and $v_0$, $u_0 \in \cH$. Moreover, assume that $\bf(A1)-\bf(A3)$ and $\bf(A4)$ hold and denote by $\tilde{u} \in \rL^2(0,\infty;\cV) \cap \rH^1(0,\infty;\cV^\ast)$ the corresponding global solution of equation \eqref{eq:model shift}. Then there exist $\mu_0$, $\delta_0>0$ such that for all $\mu>\mu_0$ and $0<\delta < \delta_0$ the problem \eqref{eq:model data} admits a unique, global solution $v \in \rL^2(0,\infty;\cV) \cap \rH^1(0,\infty;\cV^\ast)$. 
\end{prop}

\begin{proof}
For an appropriate choice of $\mu$, $\delta>0$, the operator $A + \mu  \rI_\delta$ is readily seen to satisfy $\bf(A1)$. Specifically, since for the operator $A$ assumption $\bf(A1)$ holds, invoking the bound \eqref{eq: I_delta bound2} and Young's inequality, for $v\in \cV$ we compute
    \begin{equation*}
        \begin{aligned}
            \langle (A + \mu \rI_\delta)v,v \rangle_{\cV^\ast, \cV} &= \langle Av ,v \rangle_{\cV^\ast, \cV} +  \mu \langle   \rI_\delta v - v ,v \rangle_{\cV^\ast, \cV} + \mu \langle v,v \rangle_{\cV^\ast, \cV} \\&\geq \alpha \| v \|^2_\cV -C \mu \delta \| v \|_\cH \| v \|_\cV  + (\mu-\omega) \| v \|^2_\cH \\ &\geq  \bigl (  \alpha  - \eps \bigr )\| v \|^2_\cV   + \bigl (\mu-  \frac{\mu^2\delta^2}{4\eps}-\omega\bigr )\| v \|^2_\cH.
        \end{aligned}
    \end{equation*}
    Choosing $0 < \eps < \alpha$ and  $\delta >0$ appropriately small as well as $\mu>0$ sufficiently large guarantees that 
    \begin{equation*}
         \langle (A + \mu \rI_\delta)v,v \rangle_{\cV^\ast, \cV} \geq \alpha' \| v \|^2_{\cV} + \beta \|v \|_\cH^2 \ \text{ for some }\ \alpha'>0 \ \text{ and }\ \beta \geq 0.
    \end{equation*}
    Therefore, assumption ${\bf(A1)}$ is satisfied by $A+\mu I_\delta$ with $\omega=0$. \par
    Since the non-linear term $F$ satisfies $\bf(A1)-\bf(A3)$ we obtain a unique local maximal solution of the data assimilated problem \eqref{eq:model data}. Next, we verify $\bf(A4)$. Taking inner products of \eqref{eq:model data} with $v$ and adding $\mu \langle \tilde{u}- \tilde{u}, v \rangle_{\cV^\ast, \cV}$ on the right-hand side, yields in view of the above property of $A$ and \eqref{eq: I_delta bound2}
    \begin{equation}
        \begin{aligned}
            \dt \frac{1}{2}\| v \|^2_\cH + \alpha' \| v \|^2_\cV &\leq  |\langle F(v), v\rangle_{\cV^\ast, \cV}|  + \mu|\langle \tilde{u} - \mathrm{I}_\delta \tilde{u}, v \rangle_{\cV^\ast, \cV}| +\mu|  \langle \tilde{u},v \rangle_{\cV^\ast, \cV}| \\ &\leq C \bigl ( \| F(v) \|^2_{\cV^\ast} + \| \tilde{u} \|^2_\cH + \| \tilde{u} \|^2_{\cV^\ast} \bigr ) + \eps \| v \|^2_\cV.
        \end{aligned}
    \end{equation}
    Integrating in time, using Gronwall's inequality implies in view of $\bf(A4)$ and \autoref{cor: global shift}   
    \begin{equation*}
          \lim\limits_{t \to a_{\mathrm{max}}} \| v \|_{\rL^2(0,t;\cV) \cap \rL^\infty(0,t;\cH)} < \infty
    \end{equation*}
    and therefore $v \in \rL^2_{\mathrm{loc}}(0,\infty;\cV) \cap \rH^1_{\mathrm{loc}}(0,\infty;\cV^\ast) \cap \mathrm{BUC}((0,\infty);\cH)$. By similar arguments to those in the proof of \autoref{cor: global shift} the solution can be extended to $T= \infty$.
\end{proof}
\begin{rem}
 {\rm It remains to verify that \textup{(\textbf{A4})} holds for the data-assimilated solution \(v\) in all examples below. 
Observe that the data-assimilation problem \eqref{eq:model data} can be rewritten as the forced problem \eqref{eq:model shift} with the same nonlinearity and with
\[
A' \coloneqq A + \mu\,\rI_\delta .
\]
As in the proof of \autoref{prop: global DA}, one checks that \(A'\) satisfies \textup{(\textbf{A1})}, while the additional forcing term \(\rI_\delta \tilde{u}\) is controlled by the interpolant bounds in \eqref{eq: I_delta bound2}. 
Consequently, verifying \textup{(\textbf{A4})} for the data-assimilation setting is straightforward in practice, and we shall omit this check in the examples presented in \autoref{sec: Example} and \autoref{sec: Example-weak}.
}
\end{rem}
\noindent
Having established global in time existence of a unique solution $\tilde{u}$ to problem \eqref{eq:model}, as well as the global in time existence of a unique solution $v$ to the approximate  problem \eqref{eq:model data}, it is reasonable to consider the difference $w := \tilde{u}-v$, which corresponds to the evolution equation
\begin{equation}\label{eq:model difference}
    \left\{
    \begin{aligned}
       w'+Aw&= F(\tilde{u}) - F(v) -\mu  \rI_\delta w - \omega \tilde{u}, \quad t \in (0,\infty),\\
        w(0)&=w_0.
    \end{aligned}
    \right.
\end{equation}
We are now in a position to state our main theorems concerning the long time behaviour of the difference $w$, measured in the $\cH$-and $\cV$-norm. 

\begin{thm}[Convergence in the $\cH$-norm]\label{thm: V norm}\mbox{}\\
    Let $u_0$, $v_0 \in \cH$, suppose that $\bf(A1)-\bf(A4)$ are satisfied and the observation operator $\rI_\delta$ satisfies \eqref{eq: I_delta bound2}. Then, there exist $\mu_0$, $\delta_0>0$ such that for all $\mu>\mu_0$ and $0<\delta < \delta_0$ the unique, global solution of \eqref{eq:model difference} satisfies
        \begin{equation}
            \|  w(t) \|_\cH  \to 0 \ \text{ exponentially, as } \ t \to \infty.
        \end{equation}
        In particular, if $\omega=0$ in $\bf(A1)$, then the solution of the data assimilation system \eqref{eq:model data} converges exponentially fast to the solution $u$ of the original problem \eqref{eq:model}.
\end{thm}
\noindent
We want to stress that the convergence established in \autoref{thm: V norm} is strong. Specifically, embedding $\cH \hookrightarrow \cV^\ast$ naturally leads to the following weaker convergence result.
\begin{cor}[Convergence in the $\cV^\ast$-norm]\label{cor: H norm}\mbox{}\\
    Under the same assumptions of \autoref{thm: V norm}, the unique, global solution of \eqref{eq:model difference} satisfies
        \begin{equation}
            \| w(t) \|_{\cV^\ast} \to 0 \ \text{ exponentially, as } \ t \to \infty.
        \end{equation}
\end{cor}
\begin{proof}[Proof of \autoref{thm: V norm}]
    Taking inner products of \eqref{eq:model difference} with $w$, yields 
    \begin{equation*}
        \langle \dt w, w \rangle_{\cV^\ast, \cV} + \langle A w , w \rangle_{\cV^\ast, \cV} = \langle F(\tilde u) -F(v), w \rangle_{\cV^\ast, \cV} - \mu \langle \rI_\delta w , w \rangle_{\cV^\ast, \cV}- \omega \langle \tilde{u},w \rangle_{\cV^\ast, \cV}.
    \end{equation*}
    Using the coercivity of $A$ in $\bf(A1)$, the assumptions on the pairing between $\cV^\ast$ and $\cV$ and adding the term $\mu \langle w-w,w \rangle_{\cV^\ast,\cV}$ on the right-hand side, yields 
    \begin{equation*}
        \frac{1}{2}\dt \| w \|^2_\cH + \alpha \| w \|^2_\cV \leq \langle F(u) -F(v), w \rangle_{\cV^\ast, \cV} + \mu  \langle w-\rI_\delta w , w \rangle_{\cV^\ast, \cV} - \mu \| w \|_\cH^2 - \omega \langle \tilde{u},w \rangle_{\cV^\ast, \cV}. 
    \end{equation*}
    First, we estimate using Young's inequality to obtain
    \begin{equation*}
        | - \omega \langle \tilde{u},w \rangle_{\cV^\ast, \cV} | \leq \omega \| \tilde{u} \|_{\cV^\ast} \cdot \| w \|_\cV \leq C \omega \| \tilde{u} \|^2_{\cV^\ast} + \eps \| w\|^2_\cV.
    \end{equation*}
    Next, by $\bf(A3)$, embedding $\cV_\beta \hookrightarrow \cV_{\beta_j}$ and Young's inequality we estimate 
    \begin{equation*}
        \begin{aligned}
            |\langle F(\tilde{u}) -F(v), w \rangle_{\cV^\ast, \cV} | \leq \| F(\tilde{u}) -F(v) \|_{\cV^\ast} \cdot \| w \|_\cV &\leq  C \sum_{j=1}^k \bigl ( 1+ \| \tilde{u} \|^{\rho_j}_{\cV_{\beta}} + \| v \|^{\rho_j}_{\cV_{\beta}} \bigr ) \| w \|_{\cV_{\beta_j}} \cdot \| w \|_\cV \\ &\leq C(\eps) \sum_{j=1}^k \bigl ( 1+ \| \tilde{u} \|^{2 \rho_j}_{\cV_{\beta}} + \| v \|^{ 2 \rho_j}_{\cV_{\beta}} \bigr ) \| w \|^2_{\cV_\beta} + \eps   \| w \|^2_\cV\,.
        \end{aligned}
    \end{equation*}
    Using interpolation and Young's inequalities, we obtain further
    \begin{equation*}
        \begin{aligned}
            C(\eps) \sum_{j=1}^k \bigl ( 1+ \| \tilde{u}\|^{2 \rho_j}_{\cV_{\beta}} + \| v \|^{ 2 \rho_j}_{\cV_{\beta}} \bigr ) \| w \|^2_{\cV_\beta} + \eps   \| w \|^2_\cV &\leq C(\eps) \sum_{j=1}^k \bigl ( 1+ \| \tilde{u} \|^{2 \rho_j}_{\cV_{\beta}} + \| v \|^{ 2 \rho_j}_{\cV_{\beta}} \bigr ) \| w \|^{4-4\beta}_{\cH} \| w \|^{4\beta-2}_\cV + \eps   \| w \|^2_\cV \\ &\leq C(\eps)  \sum_{j=1}^k \bigl ( 1+ \| \tilde{u} \|^{ \frac{\rho_j}{1-\beta}}_{\cV_\beta} + \| v \|^{ \frac{\rho_j}{1-\beta}}_{\cV_\beta} \bigr ) \| w \|^{2}_{\cH} +\eps \| w \|^2_\cV.
        \end{aligned}
    \end{equation*}
    Using the bound \eqref{eq: I_delta bound2}, the remaining term involving the interpolation operator can be estimated by
    \begin{equation*}
        |\mu  \langle w- \rI_\delta w , w \rangle_{\cV^\ast, \cV} | \leq  C\mu \delta \| w \|_\cH \| w \|_\cV \leq C(\eps) \mu^2 \delta^2 \| w \|^2_\cH + \eps \| w \|_\cV^2.
    \end{equation*}
    Absorbing the highest order norms, we conclude the inequality 
    \begin{equation*}
          \frac{1}{2}\dt \| w \|^2_\cH + \alpha \| w \|^2_\cV \leq
          C \bigl (  \sum_{j=1}^k \bigl ( 1+ \| \tilde{u} \|^{ \frac{\rho_j}{1-\beta}}_{\cV_\beta} + \| v \|^{ \frac{\rho_j}{1-\beta}}_{\cV_\beta} \bigr ) + \mu^2\delta^2  - \mu \bigr ) \| w \|^2_\cH + C\omega \| \tilde{u}\|^2_{\cV^\ast}
    \end{equation*}
    for a suitable constant $C>0$. Note that $\tilde{u}$ and $v$ are the given global solutions of \eqref{eq:model shift} and \eqref{eq:model data} respectively. In particular, by $\bf(A3)$ we have
    \begin{equation*}
        \rL^2(0,\infty;\cV) \cap \rH^1(0,\infty;\cV^\ast) \hookrightarrow \rH^{1-\beta}(0,\infty;\cV_\beta ) \hookrightarrow \rL^{\frac{\rho_j}{1-\beta}}(0,\infty;\cV_\beta)
    \end{equation*}
    and therefore there exist time-independent constants $C_1$, $C_2>0$ such that 
    \begin{equation*}
         \int_0^t \sum_{j=1}^k \bigl ( 1+ \| \tilde{u} \|^{ \frac{\rho_j}{1-\beta}}_{\cV_\beta} + \| v \|^{ \frac{\rho_j}{1-\beta}}_{\cV_\beta}\bigr ) \d s \leq C_1 t + C_2.
    \end{equation*}
    Hence, integrating in time, and using Gronwall's inequality yields 
    \begin{equation*}
        \frac{1}{2} \| w (t) \|_\cH^2 + \alpha \int_0^t \| w(s) \|^2_\cV \d s \leq \mathrm{e}^{C_2} \bigl ( \frac{1}{2}   \| w_0 \|^2_\cH + C\omega \|\tilde{u} \|^2_{\rL^2_t \cV^\ast} \bigr ) \cdot \mathrm{e}^{C_1(1 + \mu^2\delta^2 - \mu)t} \ \text{ for all } \ t>0.
    \end{equation*}
    The desired exponential convergence follows from choosing $\mu$, $\delta>0$ such that $1 + \mu^2\delta^2 - \mu<0$ and using $\|\tilde{u} \|^2_{\rL^2_t \cV^\ast} \leq C$ for a constant $C>0$ independent of time.
\end{proof}

\section{Illustrations of the General Framework - strong solutions}\label{sec: Example}
\noindent
As previously noted, the general framework we develop incorporates a range of models from mathematical physics, which will be discussed in detail below. We begin by presenting the general structure governing the convergence of strong solutions of the approximate system \eqref{eq:model data} to those of the original system \eqref{eq:model}, or the shifted system \eqref{eq:model shift}, depending on the value of $\omega $ from $\bf(A1)$. Let $\Omega \subset \R^n$ be an open set. Set $\cV^\ast = \rL^2(\Omega)$, $\cV= \D(A)$ and $\cH = (\rL^2(\Omega), \D(A))_{\frac{1}{2},2}$. Define the pairing between $\rL^2(\Omega)$ and $\D(A)$ by
\begin{equation}\label{eq: strong pairing}
    \langle u , v \rangle_{\rL^2(\Omega), \D(A)} \coloneqq (u,Av)_2 + (u,v)_2 \ \text{ for all } \ u \in \rL^2(\Omega) \ \text{ and} \ v \in \D(A).
\end{equation}
Here $(\cdot, \cdot)_2$ denotes the standard $\rL^2$-inner product. Concerning the interpolation operator $\mathrm{I}_\delta$, we make the standard assumption
\begin{equation}\label{eq: strong interpolation}
    \| f - \mathrm{I}_\delta f \|_2 \leq C \delta \| f \|_\cH \ \text{ for all }\ f \in \cH,
\end{equation}
see also \cite{MR3183055}. This assumption will be used throughout all examples in the strong setting. A straightforward calculation shows that this implies the bound \eqref{eq: I_delta bound2} on the measurement operator $\mathrm{I}_\delta$. Note that in contrast to \cite{MR3917673,MR3183055} we do not require any additional assumptions on the observations. We are now in a position to discuss several examples drawn from the existing literature, as well as introduce new ones that, to the best of our knowledge, have not been previously treated, thereby presenting the first data assimilation results for these cases.
\subsection{2D-Navier-Stokes equations}\label{subsec: stokes strong}\cite{MR3183055} \mbox{} \\ 
Let $\Omega \subset \R^2$ be a bounded domain with smooth boundary. The 2D-Navier-Stokes equations with no-slip boundary conditions are given by the following set of equations
   \begin{equation}\label{eq:2D Stokes}\tag{\textcolor{blue}{2D-NSE}}
    \left\{
    \begin{aligned}
        \partial_t u + (u\cdot \nabla)u - \Delta u  + \nabla p&=0 ,  &&\text{ in }(0,T) \times \Omega, \\
        \div u &= 0, &&\text{ in }(0,T) \times \Omega, \\
        u&=0,  &&\text{ in }  (0,T) \times  \partial\Omega , \\
        u(0) &=u_0,
    \end{aligned}
    \right.
\end{equation}
where $u \colon \Omega \to \R^2$ denotes the velocity field and $p \colon \Omega \to \R$ the pressure. To rewrite \eqref{eq:2D Stokes} in the form of an abstract evolution equation \eqref{eq:model}, we introduce the Stokes operator $A$ realized in the space of weakly divergence-free vector fields $\rL_\sigma^2(\Omega)$ by
\begin{equation*}
    Au \coloneqq -\mathbb{P}\Delta u, \enspace \D(A) \coloneqq  \rH^2(\Omega;\R^2) \cap \rH^1_0(\Omega;\R^2)\cap \rL^2_\sigma(\Omega;\R^2).
\end{equation*}
Here, $\mathbb{P}$ denotes the two-dimensional Helmholtz projector from $\rL^2(\Omega;\R^2)$ onto the solenoidal vector fields $\rL^2_\sigma(\Omega;\R^2)$. Then \eqref{eq:2D Stokes} can be rewritten as an abstract evolution equation \eqref{eq:model}, where $A$ is the Stokes operator and $F(u) = - (u\cdot \nabla) u$.
In the following, we verify that $\bf(A1)-\bf(A3)$ and $\bf(A4)$ are satisfied. First, we obtain in view of the orthogonality of $\mathbb{P}$ and Poincar\'e's inequality
\begin{equation*}
    \langle Au,u\rangle_{\rL^2(\Omega), \D(A)} = \| A u\|_2^2 + \| \nabla u\|^2_2 \geq \alpha \| u \|^2_{\D(A)} \ \text{ for some }\ \alpha>0 \ \text{ and all }\ u \in \D(A). 
\end{equation*}
In particular, $\bf(A1)$ is satisfied with $\omega=0$. Next, observe that the non-linear term $F$ is bilinear and therefore it suffices to verify the condition in \autoref{rem: bilinear}. Indeed, by H\"older's inequality and Sobolev embedding, we obtain
\begin{equation*}
    \| (u\cdot \nabla)u \|_2 \leq C \| u\|_4 \cdot \| \nabla u \|_4 \leq C \| u \|^2_{\rH^{\nicefrac{3}{2}}(\Omega)}.
\end{equation*}
We conclude that $\bf(A2)$, $\bf(A3)$ are satisfied with $\beta=\frac{3}{4}$ and $\rho=1$. To verify $\bf(A4)$, we note that the boundedness of the non-linear term 
is well known by the fundamental work of Ladyzhenskaya \cite{Lady59}, for all $u_0 \in \rH^1_0(\Omega;\R^2)\cap \rL^2_\sigma(\Omega;\R^2)$. For the square-integrability of the $\rH^1$-norm of $u$, we refer also to Prodi \cite[Lemma 3]{Pr:60}, where he proved the validity of the energy equality. 
\par
We denote by $v$ the approximate solution of \eqref{eq:model data}, guaranteed by \autoref{prop: global DA}. An application of \autoref{thm: V norm} then yields the following result.
\begin{cor}[Data Assimilation for \eqref{eq:2D Stokes}] \mbox{} \\
Let $u_0$, $v_0 \in \rH^1_0(\Omega;\R^2)\cap \rL^2_\sigma(\Omega;\R^2)$. Then, there exist $\mu_0$, $\delta_0>0$ such that for all $\mu>\mu_0$ and $0<\delta < \delta_0$ 
\begin{equation*}
    \| (u -v)(t) \|_{\rH^1(\Omega)}  \to 0  \ \text{ exponentially, as } \ t \to \infty.
\end{equation*}
\end{cor}
\begin{rem}
    {\rm In this subsection, we consider \eqref{eq:2D Stokes} with \emph{no-slip} boundary conditions. The analysis can also be extended to other types of boundary conditions, such as \emph{pure slip, outflow}, or \emph{free} conditions by suitably adapting the domain of the Stokes operator; see, for example, \cite[Section 7]{MR3524106}. It is important to note, however, that in these alternative settings, $\bf(A1)$ may only be valid under the additional condition that $\omega >0$.}
\end{rem}
\subsection{3D-Primitive Equations}\label{subse: primitive}\cite{Furukawa2024,MR3917673} \mbox{}\\
\noindent
Let $\Omega = \T^2 \times (0,1)\subset \R^3$. Here, $\T^2$ denotes the unit square in $\R^2$ with periodic boundary conditions. The 3D-Primitive equations are given by the following set of equations
\begin{equation}\label{eq:3D Primitive}\tag{\textcolor{blue}{3D-PE}}
    \left\{
    \begin{aligned}
        \partial_t v + (v \cdot \nablaH)v + w \cdot \dz v - \Delta v  + \nablaH p&=0 ,  &&\text{ in }(0,T) \times \Omega, \\
        \dz p &= 0, &&\text{ in }(0,T) \times \Omega, \\
        \div u &= 0, &&\text{ in }(0,T) \times \Omega, \\
        v(0) &=v_0.
    \end{aligned}
    \right.
\end{equation}
supplemented by the boundary conditions
\begin{equation}\label{eq:bcPE}
\begin{aligned} 
    v|_{\T^2 \times \{1 \}} = 0, \quad (\partial_z v)|_{\T^2 \times \{0 \}} = 0 \ \text{ and }\  w|_{\T^2 \times \{ 0\} \cup \T^2 \times \{ 1 \}} = 0.
\end{aligned} 
\end{equation}
Here $u=(v,w) \colon \Omega \to \R^3$ denotes the velocity field, $p \colon \Omega \to \R$ the pressure and we use $\nablaH$, $\divH$ for the horizontal gradient and the horizontal divergence, that is $\nablaH:=(\partial_x,\partial_y)^T$ and
$\divH:=\nablaH \cdot{}$. Note that the vertical velocity $w$ is fully determined by the divergence free condition and the boundary conditions, that is,
\begin{equation*}
    w = -\int_0^z \divH v(\cdot,\xi) \d \xi.
\end{equation*}
To rewrite \eqref{eq:3D Primitive} in the form of an abstract evolution equation, we introduce the \emph{hydrostatic Stokes operator} $A_\rH$ realized in the space of \emph{hydrostatically solenoidal vector fields} 
\begin{equation*}
    \rL^2_{\sigmabar}(\Omega) = \overline{\{v \in \rC^\infty(\overline{\Omega};\R^2) \colon \divH \vbar = 0 \}}^{\| \cdot \|_{\rL^2(\Omega)}} \ \text{ with } \ \vbar = \int_0^1 v(\cdot,\xi) \d \xi
\end{equation*}
as follows
\begin{equation*}
    A_\rH v \coloneqq -\mathbb{P}_\rH \Delta v, \enspace \D(A_\rH) \coloneqq \{ v \in \rH^2(\Omega;\R^2) \cap \rL^2_{\sigmabar}(\Omega;\R^2) \colon v|_{\T^2 \times \{1 \}} = (\partial_z v)|_{\T^2 \times \{0 \}} = 0\}.
\end{equation*}
Here, $\mathbb{P}_\rH$ denotes the \emph{hydrostatic Helmholtz projection}. For an extensive discussion of the hydrostatic Stokes operator and related results we refer to \cite{HK:16}. Hence we write \eqref{eq:3D Primitive} in the form of an abstract evolution equation \eqref{eq:model} by choosing $A$ to be the hydrostatic Stokes operator and setting $F(v) = -(v\cdot \nablaH)v - w \cdot \dz v$. By orthogonality of $\mathbb{P}_\rH$ and Poincar\'e's inequality, we verify
\begin{equation*}
     \langle A_\rH v,v\rangle_{\rL^2(\Omega), \D(A_\rH)} = \| A_\rH v\|_2^2 + \| \nabla v\|^2_2 \geq \alpha \| v \|^2_{\D(A_\rH)} \ \text{ for some }\ \alpha>0 \ \text{ and all }\ v \in \D(A_\rH), 
\end{equation*}
which implies that $\bf(A1)$ is satisfied with $\omega=0$. Moreover, arguing as in \cite[Section 5]{HK:16}, we see that $\bf(A2)$ and $\bf(A3)$ are satisfied with $\beta=\frac{3}{4}$ and $\rho=1$.
Finally, $\bf(A4)$ is satisfied by \cite[Section 6]{HK:16} for all $v_0 \in \rH^1(\Omega;\R^2) \cap \rL^2_{\sigmabar}(\Omega;\R^2)$ subject to suitable compatibility conditions, that are specified in \autoref{cor: DA Prim}. Denoting by $\hat{v}$ the solution of the approximate system \eqref{eq:model data}, guaranteed by \autoref{prop: global DA}, an application of \autoref{thm: V norm} yields the following result.
\begin{cor}[Data assimilation of \eqref{eq:3D Primitive}]\label{cor: DA Prim} \mbox{} \\
Let $v_0$, $\hat{v}_0 \in \rH^1(\Omega;\R^2) \cap \rL^2_{\sigmabar}(\Omega;\R^2)$ satisfying the compatibility conditions $v|_{\T^2 \times \{1\}} = \hat{v}|_{\T^2 \times \{1\}} =0$. Then, there exist $\mu_0$, $\delta_0>0$ such that for all $\mu>\mu_0$ and $0<\delta < \delta_0$ 
\begin{equation*}
    \| (v -\hat{v})(t) \|_{\rH^1(\Omega)} \to 0  \ \text{ exponentially, as } \ t \to \infty.
\end{equation*} 
\end{cor}
\subsection{Energy Balance Model coupled to an active fluid}\label{subsec:EBM}\cite{DHT:25, DHPT:25} \mbox{}\\ 
\noindent
Energy Balance Models (EBMs) are conceptual climate models that describe the evolution of the Earth's temperature based on the fundamental principle of radiative balance. They are instrumental for investigating core climate dynamics, such as bistability and ice-albedo feedback, by capturing essential physics without the computational complexity of general circulation models (\cite{Budyko1969,Sellers1969,North17}).

Let $\Omega = \T^2 \times (0,1) \subset \R^3$ be as in \autoref{subse: primitive}. We consider the following system of equations, introduced in \cite{DHT:25, DHPT:25} 
\begin{equation}
\left\{
\begin{aligned}
\partial_t v + (v \cdot \nablaH) v + w \cdot \dz v - \Delta v + \nablaH p &= \int_0^z \nablaH \tau(\cdot,\xi) \d \xi,  &&\quad \text{in }  (0,T) \times \Omega   ,\\
\partial_z p &= 0,  &&\quad \text{in } (0,T) \times \Omega ,\\
\div u &= 0, &&\quad \text{in } (0,T) \times \Omega ,\\
\partial_t \tau+ (v \cdot \nablaH) \tau + w \cdot \dz \tau - \Delta \tau &= 0,  &&\quad \text{in } (0,T) \times  \Omega ,\\
\tau|_{\T^2 \times \{1\}} &= \rho,  &&\quad \text{in } (0,T) \times \T^2,\\
\partial_t \rho + (\vbar \cdot \nablaH) \rho - \DeltaH \rho + (\partial_z \tau)|_{\T^2 \times \{1\}} &= Q(t,x)\,\beta(\rho) - |\rho|^3 \rho,  &&\quad \text{in } (0,T) \times \T^2,\\
v(0) &= v_0, \qquad \tau(0)= \tau_0,
\end{aligned}
\right.
\label{eq: primitive + EBM}\tag{EBM-PE}
\end{equation}
where $v_0$ and $\tau_0$ are the initial conditions, supplemented by the boundary conditions
\begin{equation}\label{eq:bcEBM}
\begin{aligned} 
      v|_{\T^2 \times \{1 \}} = 0, \quad (\partial_z v)|_{\T^2 \times \{0 \}} = 0, \quad   w|_{\T^2 \times \{ 0\} \cup \T^2 \times \{ 1 \}} = 0 \ \text{ and } \ \tau|_{\T^2 \times \{0\}}=0.
\end{aligned} 
\end{equation}
Here, $v \colon \Omega \to \R^2$ denotes the velocity field, $p\colon \Omega \to \R$ denotes the pressure, $\tau \colon \Omega \to \R$ denotes the temperature and $\rho = \tau|_{\T^2 \times \{1\}} \colon \T^2 \to \R$ denotes the temperature evaluated at the surface. Moreover, $Q \in \rC^1_b(\R_+ \times \T^2)$ represents the positive solar radiation and $\beta$ is the Lipschitz continuous co-albedo, hence resulting in a Sellers-type EBM, which is parametrised by
\begin{equation*}
\beta(\rho) = \beta_1 + (\beta_2 - \beta_1)\, \frac{1+\tanh(\rho-\rho_{\mathrm{ref}})}{2},
\end{equation*}
with \(0<\beta_1<\beta_2\) corresponding to the co-albedo values for ice-covered and ice-free conditions, respectively, and \(\rho_{\mathrm{ref}}\) being the temperature at which ice becomes white. For more details we refer to \cite{DHT:25}.
To rewrite \eqref{eq: primitive + EBM} as the abstract evolution equation \eqref{eq:model}, we introduce the operator matrix $A$ realized in the space $\cV^\ast = \rL^2_{\sigmabar}(\Omega;\R^2) \times \rL^2(\Omega) \times \rL^2(\T^2)$ by
\begin{equation*}
    A \coloneqq \begin{pmatrix}
        A_\rH & 0 & 0 \\ 
        0 & -\Delta & 0 \\ 
        0 & \gamma \dz & -\DeltaH \\
    \end{pmatrix}, \enspace \D(A) = \D(A_\rH) \times   \{ (\tau,\rho) \in \rH^{2}(\Omega) \times \rH^{2}(\T^2) \ \colon \tau|_{\T^2 \times \{1\}}= \rho, \ \tau|_{\T^2 \times \{0\}} =0  \},
\end{equation*}
where $A_\rH$ denotes the hydrostatic Stokes operator as introduced in \autoref{subse: primitive} and $\gamma$ denotes the trace operator. Note that with this choice of boundary conditions, the operator $A$ is invertible  due to its lower-triangular structure. Next, set
\begin{equation*}
    F\bigl ((v,\tau,\rho)^\top \bigr ) = \begin{pmatrix}
        - (v \cdot \nablaH v) - w \cdot \dz v +\int_0^z \nablaH \tau(\cdot,\xi) \d \xi \\
        - (v \cdot \nablaH \tau) - w \cdot \dz \tau \\
        -  (\vbar \cdot \nablaH) \rho + Q(t,x) \beta(\rho) - |\rho|^3 \rho.
    \end{pmatrix}
\end{equation*}
In view of integration by parts and the invertibility of $A$, we calculate for all $\mathbf{x}= (v,\tau,\rho) \in \D(A)$
\begin{equation*}
    \begin{aligned}
         \langle A\mathbf{x}, \mathbf{x} \rangle_{\cV^\ast, \D(A)} = \| A \mathbf{x} \|_{\cV^\ast}^2
 + \| \nabla v \|_2^2 + \| \nabla \tau \|^2_2 + \| \nablaH \rho\|^2_2 \geq \alpha \| \mathbf{x} \|^2_{\D(A)}  \ \text{ for some }\ \alpha>0. 
    \end{aligned}
\end{equation*}
Specifically, $\bf(A1)$ is valid with $\omega=0$. Moreover, \cite[Lemma 5.5]{DHT:25} implies that $\bf(A2)$ and $\bf(A3)$ are satisfied. Finally, $\bf(A4)$ is valid by \cite[Theorem 4.1]{DHT:25} for all initial data $\mathbf{x}_0 = (v_0,\tau_0,\rho_0)$ satisfying
\begin{equation}\label{eq: ass data EBM}
    v_0 \text{ as in \autoref{cor: DA Prim}} \ \text{ and }\ \tau_0 \in \rH^1(\Omega)  \ \text{ such that }\ \rho_0 = \tau_0 |_{\T^2 \times \{1\}}\in \rH^1(\T^2).
\end{equation}
To verify that the $\rH^1$-norm of $\mathbf{x}(t)$ is square-integrable on $(0,\infty)$, we first take the inner product of \eqref{eq: primitive + EBM}$_4$ with $\tau$ and use H\"older's and Young's inequality to obtain 
\begin{equation*}
    \frac{1}{2}\dt \bigl ( \| \tau \|^2_2 + \| \rho \|^2_2 \bigl ) + \| \nabla \tau \|^2_2 + \| \nablaH \rho \|^2_2 + \| \rho\|^5_5 = \int_{\T^2} Q(t,x) \beta(\rho) \cdot \rho \leq C \| Q(t,x)  \|_{\nicefrac{5}{4}}^{\nicefrac{5}{4}} + \eps \| \rho \|^5_5.
\end{equation*}
Integrating in time, using Gronwall's inequality and assuming $Q(t,x) \in \rL^{\nicefrac{5}{4}}(\R_+ \times \T^2)$ yields the result $\| \tau \|^2_{\rH^1}$ and $\| \rho \|^2_{\rH^1}$ are integrable on $(0,\infty)$. Moreover, taking the inner product of \eqref{eq: primitive + EBM}$_1$ with $v$ and arguing similarly implies that $\| v \|^2_{\rH^1}$ is integrable on $(0,\infty)$. Denoting by $\hat{\mathbf{x}} = ( \hat{v}, \hat{\tau}, \hat{\rho})$ the solution of the approximate system \eqref{eq:model data}, \autoref{thm: V norm} implies the following convergence result.
\begin{cor}[Data assimilation for \eqref{eq: primitive + EBM}] \mbox{}\\
Let $\mathbf{x}_0$ and $\hat{\mathbf{x}}_0$ satisfy \eqref{eq: ass data EBM}. Then, there exist $\mu_0$, $\delta_0>0$ such that for all $\mu>\mu_0$ and $0<\delta < \delta_0$ 
\begin{equation*}
    \| (\mathbf{x} -\hat{\mathbf{x}})(t) \|_{\rH^1(\Omega) \times \rH^1(\Omega) \times \rH^1(\T^2)}  \to 0  \ \text{ exponentially, as } \ t \to \infty.
\end{equation*} 
\end{cor}

\subsection{2D-Bidomain problem with FitzHugh-Nagumo transport}\label{subsec:bidomain} \mbox{}\\
\noindent The bidomain model offers a continuum framework for the electrical activity in cardiac tissue, representing it as overlapping intra and extracellular domains. Coupled with FitzHugh-Nagumo kinetics, which provide a simplified yet powerful caricature of action potential dynamics, the system is a cornerstone of computational cardiology for simulating cardiac wave propagation and investigating the mechanisms of arrhythmia. For an exhaustive overview of the model, we refer to \cite{Keener,ColliFranzone}. 

We assume that $\Omega\subset\mathbb{R}^2$ is a bounded domain with smooth boundary. The system we will focus on reads as 
\begin{equation}
\left\{
\begin{aligned}
      \partial_t u-\div(a_1\nabla u_1)&=f(u,w),  &&\text{in }(0,T) \times \Omega,\\
        \partial_t u+\div(a_2\nabla u_2)&=f(u,w), &&\text{in }(0,T) \times \Omega,\\
        \partial_tw&=g(u,w), &&\text{in }(0,T) \times \Omega , \\
        u_1-u_2&=u,&&\text{in }(0,T) \times \Omega,\\
        u(0)&=u_0, \quad w(0)=w_0.
\end{aligned}
\right.
\label{eq:bid}\tag{2D-BIDOMAIN}
\end{equation}
endowed with boundary conditions:
\begin{equation} \label{eq:BC-Bid}
a_i\nabla u_i \cdot \nu=0 \ \text{ on }\ \partial \Omega \times (0,T), \  i=1,2\,,
\end{equation}
where $\nu$ is the outer unit normal vector of the surface $\partial\Omega$. Here, $(u,w): \Omega\to\mathbb{R}^2$. Moreover, we assume that $a_i=a_i(x)$, $i=1,2$, belong to $\rW^{1,\infty}(\Omega;\mathbb{R}^{2\times2})$, are uniformly positive definite on $\Omega$ and there exists $\gamma\in \rH^{1}(\partial\Omega)$ such that $
\nu(x)\cdot a_2(x)=\gamma(x)\nu(x)\cdot a_1(x)\,.
$
In particular, we deduce that $\gamma(x)\ge\gamma_0>0$, for all $x\in\partial\Omega$. 
Concerning the non-linear right-hand sides $f$ and $g$, we assume they are transport terms of FitzHugh-Nagumo type given by
\[
\begin{aligned}
f(u,w):=-u^3+(a+1)u^2 -(a+\delta)u - w\ \text{ and } \ g(u,w):=-bw+cu \ \text{ for } \ a  \in (0,1) \ \text{ and } \ b,c,   \delta >0.
\end{aligned}
\]
To rewrite system \eqref{eq:bid} as an abstract evolution equation \eqref{eq:model}, we first introduce the bidomain operator $\mathbb{A}$, realized in $\rL_0^2(\Omega)$, the space of mean value free functions, by
\begin{equation*}
    \mathbb{A}:=\bigl ( \div(a_1 \nabla)^{-1} + \div(a_2 \nabla)^{-1} \bigr)^{-1}  , \ \D(\mathbb{A}):= \{ u \in \rH^2(\Omega) \cap \rL^2_0 (\Omega)\colon \nu \cdot a_1 \nabla u =\nu \cdot a_2 \nabla u =0  \text{ on }  \partial \Omega  \},
\end{equation*}
and extend its definition to $\rL^2(\Omega)$. We then define the operator matrix $A$, realized in the space $\cV^\ast =\rL^2(\Omega)\times\rL^2(\Omega)$, by
\begin{equation*}
A:=\begin{pmatrix}
    \varepsilon+\mathbb{A} & 0 \\\  0 &b
\end{pmatrix}, \  \D(A) = \D(\mathbb{A}) \times \rL^2(\Omega)
\end{equation*}
and set
\begin{equation*} F(u,w)=
    (-u^3+(a+1)u^2 -w-(a-\varepsilon+\delta) u, cu ).
\end{equation*}
For more details, we refer to \cite{MR3834925}. To verify $\bf(A1)$, we calculate for $\mathbf{x}= (u,v) \in \D(A)$
\begin{equation*}
    \langle A \mathbf{x}, \mathbf{x}\rangle_{\cV^\ast, \D(A)} = \| A \mathbf{x} \|^2_{\cV^\ast} + \| (\mathbb{A}+\varepsilon)^{\nicefrac{1}{2}} u \|^2_2  +b\| w \|^2_2 \geq \alpha \| \mathbf{x} \|^2_{\D(A)} \ \text{ for some } \ \alpha >0,
\end{equation*}
where we used the relation $ \| (\mathbb{A}+\varepsilon)^{\nicefrac{1}{2}} u \|^2_2 \geq C \| u \|^2_{\rH^1}$ in the last step. Next, as a polynomial of order $3$, $F$ naturally satisfies $\bf(A2)$ and $\bf(A3)$ with $\beta = \nicefrac{1}{3}$ and $\rho=2$, see also \cite{MR3834925}. In particular, \cite[Theorem~4.1]{MR3834925} guarantees the non blow-up condition in $\bf(A4)$. To verify that the $\rH^1$-norm of the solution is square-integrable on $(0,\infty)$,
we take inner products with $u$ and $c^{-1}w$ respectively and adding the resulting equations yields
\begin{equation*}
\frac{1}{2}\dt  \bigl ( \|u\|_2^2+\frac{1}{c}\|w\|_2^2 \bigr )+ \| \bigl (\mathbb{A}+\eps \bigr )^{\nicefrac{1}{2}} u \|^2_2+(a-\varepsilon+\delta) \| u \|^2_2+\frac{b}{c}\|w\|_2^2 +  \| u \|^4_4 = (a+1) \| u \|^3_3.
\end{equation*}
Using the relation $\| \bigl (\mathbb{A}+\eps\bigr )^{\nicefrac{1}{2}} u \|^2_2 \geq C \| u \|^2_{\rH^1}$ and the estimate,
\begin{equation*}
    (a+1) \| u\|^3_3 \leq (a+1) \|u \|_2 \cdot \|u \|^2_4 \leq \frac{(a+1)^2}{4}\| u \|^2_2 +  \| u \|^4_4 
\end{equation*}
which follows from H\"older and Young inequalities. We conclude that 
\begin{equation*}
    \frac{1}{2}\dt  \bigl ( \|u\|_2^2+\frac{1}{c}\|w\|_2^2 \bigr )+ C \| u \|^2_{\rH^1} +\frac{b}{c}\|w\|_2^2+\biggl(a-\varepsilon+\delta-\frac{(a+1)^2}{4}\biggr)\|u\|_2^2\leq 0.
\end{equation*}
We make the following assumption:
\[
a-\varepsilon+\delta-\frac{(a+1)^2}{4}\geq 0\,.
\]
Denoting by $\hat{\mathbf{x}} = ( \hat{u}, \hat{w}) $ the solution to the approximated system \eqref{eq:model data}, guaranteed by \autoref{prop: global DA}, an application of \autoref{thm: V norm} implies the following convergence result for the solution $\mathbf{x}$ of the $\eps$ shifted system, for each $\eps>0$.
\begin{cor}[Data assimilation for \eqref{eq:bid}] \mbox{}\\
Let ${\mathbf{x}}_0$, $\hat{\mathbf{x}}_0 \in \rH^1(\Omega) \times \rL^2(\Omega)$. Then, there exist $\mu_0$, $\delta_0>0$ such that for all $\mu>\mu_0$ and $0<\delta < \delta_0$ 
\begin{equation*}
    \| (\mathbf{x} -\hat{\mathbf{x}})(t) \|_{\rH^1(\Omega) \times \rL^2(\Omega)}  \to 0  \ \text{ exponentially, as } \ t \to \infty.
\end{equation*} 
\end{cor}

\section{Illustrations of the General Framework - weak solutions}\label{sec: Example-weak}
\noindent
In this final section, we present the general framework for analysing the convergence of weak solutions of \eqref{eq:model data} to those of \eqref{eq:model}. To this end, we revisit the system \eqref{eq:2D Stokes} and introduce new examples. Let $\Omega \in \R^n$ be an open set. Set $\cH = \rL^2(\Omega)$ and assume the pairing between $\cV^\ast$ and $\cV$ satisfies
\begin{equation}
    \label{eq: pairing weak}
    \langle u,v \rangle_{\cV^\ast, \cV}= \langle u,v\rangle_\cH = (u,v)_2 \ \text{ for all }\ u \in \rL^2(\Omega) \ \text{ and }\ v\in \cV.
\end{equation}
Concerning the interpolation operator $\mathrm{I}_\delta$, we impose 
\begin{equation}\label{eq: weak interpolation}
    \| f - \mathrm{I}_\delta f \|_{\cV^\ast} \leq C \delta \| f \|_2 \ \text{ for all }\ f \in \rL^2(\Omega),
\end{equation}
Note that the condition for the measurements $\mathrm{I}_\delta$ is relaxed compared to the strong setting. However, in most practical examples, such as those in \cite{MR4344886, MR4409797} the stronger condition \eqref{eq: strong interpolation} is assumed.
\subsection{2D-Navier-Stokes equations - revisited} \mbox{} \\ 
Consider the variational formulation of the \eqref{eq:2D Stokes} on a bounded domain $\Omega$, whose strong formulation was discussed in \autoref{subsec: stokes strong}. 
For this purpose, let $\phi \in \rH^1_0(\Omega;\R^2)$ with $\div \phi =0$ and define the weak Stokes operator by
\begin{equation*}
    A_w \colon \D(A_w) := \rH^1_0 (\Omega;\R^2) \cap \rL^2_\sigma (\Omega;\R^2) \to \rH^{-1}_\sigma (\Omega;\R^2), \ \langle A_w u , \phi \rangle_{\rH^{-1}_\sigma, \D(A_w)} := (\nabla u, \nabla \phi)_2 \ \text{ for all }\ u \in \D(A_w). 
\end{equation*}
Setting 
\begin{equation*}
    \langle F_w (u),\phi \rangle_{\rH^{-1}_\sigma, \D(A_w)} := (u\otimes u , \nabla \phi)_2
\end{equation*}
results in the variational formulation 
\begin{equation}\label{eq: variational Stokes}\tag{weak-2D-Stokes}
    \left\{
    \begin{aligned}
       u'+A_w u&=F_w(u) , \quad t \in (0,T) ,\\
        u(0)&=u_0.
    \end{aligned}
    \right.
\end{equation}
To verify $\bf(A1)$, we calculate for all $u \in \D(A_w)$  in view of Poincar\'e's inequality 
\[
\langle A u, u \rangle_{\rH^{-1}_\sigma,\D(A_w)}=\|\nabla u\|_2^2 \geq \alpha \| u \|_{\D(A_w)}^2 \ \text{ for some }\ \alpha >0.
\]
Moreover, by H\"older's inequality and Sobolev embedding we obtain
\begin{equation*}
    (u\otimes u , \nabla \phi)_2 \leq C\| u\otimes u \|_2 \cdot \| \phi \|_{\rH^1} \leq C\| u\|^2_4 \cdot \| \phi \|_{\rH^1} \leq C\| u \|^2_{\rH^{\nicefrac{1}{2}}} \cdot \| \phi \|_{\rH^1}.
\end{equation*}
Hence, $\bf(A2)$ and $\bf(A3)$ are satisfied with $\beta=\frac{3}{4}$ and $\rho=1$. Since $u \in \D(A_w)$ is an eligible test function, we can test the variational formulation of \eqref{eq:2D Stokes} with $u$, resulting in the energy equality
\begin{equation*}
    \frac{1}{2}\| u(t)\|^2_2 + \int_0^t \| \nabla u\|_{2}^2 \d s = \frac{1}{2}  \| u_0 \|^2_2 \ \text{ for all }\ t \in (0,\infty).
\end{equation*}
This readily implies that $\bf(A4)$ is valid. Denoting by $v$ the approximate solution of \eqref{eq:model data}, guaranteed by \autoref{prop: global DA}, an application of \autoref{thm: V norm} then yields the following result.
\begin{cor}[Data Assimilation for \eqref{eq: variational Stokes}] \mbox{} \\
Let $u_0$, $v_0 \in \rL^2_\sigma(\Omega;\R^2)$. Then, there exist $\mu_0$, $\delta_0>0$ such that for all $\mu>\mu_0$ and $0<\delta < \delta_0$ 
\begin{equation*}
    \| (u -v)(t) \|_{\rL^2(\Omega)}  \to 0  \ \text{ exponentially, as } \ t \to \infty.
\end{equation*}
\end{cor}
\subsection{1D-Allen-Cahn equation}\cite{AllenCahn} \mbox{} \\
The Allen-Cahn equation is a seminal reaction-diffusion model that arises from a Ginzburg-Landau free energy functional to describe phase separation processes. It is a prototypical example of a system featuring front propagation, where the interface motion is driven by mean curvature, making it a fundamental testbed for studying pattern formation and metastability in materials science (\cite{AllenCahn}).

Let $\Omega =(0,1) \subset \R$. Consider the one-dimensional Allen-Cahn model given by the equations
\begin{equation}
\left \lbrace
    \begin{aligned}
        \partial_t u & = \partial_{xx} u  + u- u^3, \quad &&\text{in } (0,T) \times \Omega , \\
        u &= 0, \quad &&\text{in }  (0,T) \times \partial \Omega  ,\\
        u(0) &= u_0, 
    \end{aligned}
    \right.
    \label{eq: 1D ACE}\tag{1D-Allen-Cahn}
\end{equation}
To reformulate \eqref{eq: 1D ACE} as an abstract evolution equation \eqref{eq:model} we define the weak Dirichlet Laplacian by
$$
\partial_{xx}^w \colon \D(\partial_{xx}^w) = \rH^1_0(\Omega) \to \rH^{-1}(\Omega), \enspace \langle \partial_{xx}^w u, \phi \rangle_{\rH^{-1},\D(\partial_{xx}^w)} = ( \partial_x u , \partial_x \phi)_2 \ \text{ for all } \ u,\phi \in \D(\partial_{xx}^w)
$$
and set 
$$\langle F_w(u), \phi \rangle_{\rH^{-1},\D(\partial_{xx}^w)} := (u,\phi)_2 - (u^3,\phi)_2 \ \text{ for all } \ u,\phi \in \D(\partial_{xx}^w).$$
To verify $\bf(A1)$, note that by Poincar\'e's inequality we have
$$
\langle  \partial_{xx}^w u, u \rangle_{\rH^{-1},\D(\partial_{xx}^w)} = \| \partial_x u \|^2_2  \geq \alpha \| u \|^2_{\rH^1} \ \text{ for all } \ u \in \D(\partial_{xx}^w),
$$
for a constant $\alpha >0$. Regarding $\bf(A2)$, $\bf(A3)$, we estimate using H\"older's inequality and the embedding $ \rL^r(\Omega) \hookrightarrow \rH^{-1}(\Omega) $, which holds for any $r \in (1,\infty)$
\begin{equation*}
\begin{split}
        \| F(u) - F(v) \|_{\rH^{-1}} \leq C \bigl ( \| u-v \|_2 + \| (u^2 + v^2) \abs{u-v}\|_r \bigr)  
        & \leq  C\bigl ( \| u-v \|_2+(  \|u \|^2_{3r} + \| v \|_{3r}^2) \| u-v\|_{3r} \bigr).    
\end{split}
\end{equation*}
In particular, $\bf(A1)$, $\bf(A2)$ are valid for $\beta=2/3$, $\rho=2$ and $r\in (1,2)$ arbitrary. To verify $\bf(A4)$, we test \eqref{eq: 1D ACE} with $u$, which is an eligible test function, to obtain 
\begin{equation*}
    \frac{1}{2}\dt \| u \|^2_2 + \| \partial_x u \|^2_2 + \| u \|^4_4 = \| u \|^2_2 \leq \kappa \| \partial_x u \|^2_2\ \text{ with }\ \kappa \in (0,1)
\end{equation*}
where we used Poincar\'e's inequality in the last step; note that $\kappa \in (0,1)$ since $\kappa = \lambda_1^{-1}$, where $\lambda_1 = \pi^2$ is the smallest eigenvalue of the Dirichlet Laplacian on $(0,1)$, see \cite{Evans2010}. Integrating in time and Gronwall's inequality then imply that $\bf(A4)$ is valid. Denoting by $v$ the approximate solution of \eqref{eq:model data}, guaranteed by \autoref{prop: global DA}, an application of \autoref{thm: V norm} then yields the following result.
\begin{cor}[Data Assimilation for \eqref{eq: 1D ACE}] \mbox{} \\
Let $u_0$, $v_0 \in \rL^2(\Omega)$. Then, there exist $\mu_0$, $\delta_0>0$ such that for all $\mu>\mu_0$ and $0<\delta < \delta_0$ 
\begin{equation*}
    \| (u -v)(t) \|_{\rL^2(\Omega)}  \to 0  \ \text{ exponentially, as } \ t \to \infty.
\end{equation*}
\end{cor}
\subsection{1D and 2D Cahn-Hilliard equation}\cite{CahnHilliard} \mbox{}\\
The Cahn-Hilliard equation is a foundational fourth-order partial differential equation that models phase separation in binary mixtures while, crucially, conserving the total mass of each component. It describes complex coarsening dynamics, such as spinodal decomposition, where a system lowers its free energy by forming distinct phase domains, making it indispensable in materials science and the study of fluid mixtures (\cite{CahnHilliard}).

Let $\Omega \subset \mathbb{R}^d$, with $d =1,2,$ be an open and bounded domain with a smooth boundary. We consider the $d$-dimensional Cahn-Hilliard equation
\begin{equation}
    \left \lbrace
    \begin{aligned}
        \partial_t u + \Delta^2 u &= \Delta f(u), \quad &\text{ in } &(0,T) \times  \Omega,\\
        \nabla u \cdot \nu & = 0, \quad  & \text{ in } &(0,T) \times \partial \Omega, \\
         \nabla (\Delta u) \cdot \nu  &= 0, & \text{ in } &(0,T) \times \partial \Omega, \\
         u(0)& = u_0,
    \end{aligned}
    \right.
    \label{eq: Cahn-Hilliard}
\end{equation}
where $\nu$ denotes the outward unit vector. The nonlinear term is given by $f(u) = u^3 -u$, which corresponds to the derivative of the standard double-well potential, but more general $f$ can be chosen, see for instance \cite[Section 5.1]{AgrestiVeraar24}. To cast \eqref{eq: Cahn-Hilliard} in the abstract evolution framework \eqref{eq:model}, we consider the spaces
$$
\cV =  \rH^2_N(\Omega) =\left \lbrace  u \in \rH^2(\Omega) \; | \; \nabla u \cdot \nu = 0, \; \nabla (\Delta u) \cdot \nu = 0 \text{ in }\partial \Omega\right \rbrace, \quad \cH = \rL^2(\Omega), \quad  \cV^\ast =  \left(\rH^2_N(\Omega) \right)^\ast.
$$
The linear operator $A \colon \cV \to \cV^\ast $ and the nonlinear operator $F \colon \cV \to \cV^\ast $ are given by
$$
\langle Au, v \rangle_{\cV^\ast, \cV } = \int_{\Omega} \Delta u \,  \Delta v  \, dx, \quad \langle F(u), \phi \rangle_{\cV^\ast, \cV }  = \int_{\Omega} ( \Delta f (u)) \,  \phi \,  dx.
$$
We now verify the necessary assumptions on $A$ and $F$. The operator $A$ is quasi-coercive. Indeed, by elliptic regularity theory for the Laplacian operator with Neumann boundary conditions, we find that there exists $\alpha>0$ such that 
$$
\langle Au, u \rangle_{\cV^\ast, \cV} = \| \Delta u\|_2^2 \geq \alpha \| u \|_{\cV}^2 - \| u \|_{\cH}^2.
$$
Thus $\bf(A1)$ holds with $\omega = 1$. 

To verify the local Lipschitz condition for $F$, we estimate $\| F(u_1)- F(u_2) \|_{\cV^\ast}$. Using integration by parts and the boundary conditions, we have
\begin{equation*}
\begin{split}
    \| F(u_1)- F(u_2) \|_{\cV^\ast} &= \sup_{ \| \phi \|_\cV \leq 1}  \abs{\langle  F(u_1)- F(u_2), \phi \rangle } \\
    & = \sup_{ \| \phi \|_\cV \leq 1} \abs{ \int_{\Omega} \nabla (f(u_1)- f(u_2)) \cdot \nabla \phi dx}\\
    &=\sup_{\| \phi \|_\cV \leq 1} \abs{ \int_{\Omega} (f(u_1)- f(u_2)) \cdot \Delta \phi \, dx}.
    \end{split}
\end{equation*}
Now, observe that, for the specific choice $f(u) = u^3 -u$ we can find a bound for the difference 
$$
\abs{f(u_1) - f(u_2)} \leq C (1+ u_1^2 + u_2^2)  \abs{u_1 - u_2},
$$
where $C$ is a positive constant that will change line by line. Applying this bound, the triangle inequality and H\"older's inequality yields  
\begin{equation*}
\begin{split}
    \| F(u_1)- F(u_2)\|_{\cV^\ast} &\leq C \sup_{ \| \phi \|_\cV \leq 1} \int_{\Omega} (1+ \abs{u_1}^2 + \abs{u_2}^2)  \abs{ u_1 - u_2}  \abs{ \Delta \phi}  \, dx \\
    & \leq  \| (1+ \abs{u_1}^2 + \abs{u_2}^2)  ( u_1 - u_2) \|_2 \\
    &\leq \|1+  \abs{u_1}^2 + \abs{u_2}^2 \|_{p_1}  \| u_1 - u_2 \|_{p_2}\\
    &\leq C ( 1+ \| u_1 \|_{2p_1}^2 + \|u_2\|_{2p_1}^2)   \| u_1 - u_2 \|_{p_2},
    \end{split}
\end{equation*}
where $\frac{1}{2} = \frac{1}{p_1}+ \frac{1}{p_2}$. Choosing $p_1 = 3$ and $p_2 = 6$, we obtain
$$
 \| F(u_1)- F(u_2)\|_{\cV^\ast} \leq C ( 1+ \| u_1 \|_{6}^2 + \|u_2\|_{6}^2)   \| u_1 - u_2 \|_{6}. 
$$
By Sobolev embedding theorems, for the interpolation space $\mathcal{V}_{\beta} = [\mathcal{V}^*, \mathcal{V}]_\beta \hookrightarrow H^{4\beta - 2}(\Omega)$, we have $\mathcal{V}_\beta \hookrightarrow L^6(\Omega)$ for $\beta \geq 7/12$ if $d=1$, and for $\beta \geq 2/3$ if $d=2$. Thus, choosing $\beta_j = \beta = 2/3$ and $\rho=2$, both \textbf{(A2)}-\textbf{(A3)} are satisfied. Consequently, by \autoref{lem: local}, there exists a unique local solution to \eqref{eq: Cahn-Hilliard} in the class $\rL^2_t(\rH^2_x) \cap \rH^1_t(\rH^{-2}_x) \cap \mathrm{BUC}_t(\rL^2_x)$ on a maximal time interval $[0, t_+(u_0))$.

To establish global existence, we consider the shifted problem
\begin{equation}
\left \lbrace
    \begin{aligned}
        \Tilde{u}' + (A+ I) \Tilde{ u} &= F(\Tilde{u}), \\
         \Tilde{u}(0) &= u_0,
    \end{aligned}
    \label{eq: shifted Cahn Hilliard}
    \right.
\end{equation}
and verify the conditions of $\bf(A4)$. The functional
$$
E[\Tilde{u}] = \int_{\Omega}  \frac{1}{2}\abs{ \nabla \Tilde{u}}^2 + \frac{1}{4} \abs{\Tilde{u}}^4 \, dx
$$
is a Lyapunov functional for \eqref{eq: shifted Cahn Hilliard}. In other words, if $u  \in \rL^2_t (\rH^2_x) \cap \rH^1_t(\rH^{-2}_x) \cap \mathrm{BUC}_t(\rL^2_x) $ denotes a maximal solution of \eqref{eq: shifted Cahn Hilliard}, then it can be checked that
$$
\dt  E[\Tilde{u}(t)] \leq  - \| \nabla \varphi (t) \|_2^2- \| \Tilde{u}(t) \|_4^4,
$$
where $\varphi := f(u) - \Delta u.$ Following the previous computations, we have
$$
\| F(\Tilde{u}(t)) \|_{\cV^\ast} \leq \| f(\Tilde{u}(t))\|_2 = \| \Tilde{u}^3(t) - \Tilde{u}(t) \|_2 \leq \| \Tilde{u}(t)\|_6^3 + \|\Tilde{u}(t)\|_2 .
$$
By the embeddings $\rH^2(\Omega) \hookrightarrow \rL^6(\Omega) \hookrightarrow \rL^2(\Omega) $, which hold for $d =1,2$ in the bounded domain $\Omega$, we deduce, thanks to the Lyapunov functional, that $ \lim \limits_{t \to t_+}\| F(\Tilde{u}(t)) \|_{\cV^\ast} < \infty$. 

Next, we verify that $\| \Tilde{u}(t) \|_2^2 \in \rL^1(0,\infty).$ Taking the inner product of \eqref{eq: shifted Cahn Hilliard}$_1$ with $\Tilde{u}$ gives
$$
\frac{1}{2}\dt\|  \Tilde{u}(t)\|_2^2 + \| \Delta \Tilde{u}(t) \|_2^2 + \| \Tilde{u}(t) \|_2^2 = (\Delta f(\Tilde{u}), \Tilde{u})_2. 
$$
Integrating the right-hand side by parts, we get
$$
 (\Delta f(\Tilde{u}), \Tilde{u})_2 = -( f'(\Tilde{u}) \nabla \Tilde{u}, \nabla \Tilde{u})_2 = ((1- 3\Tilde{u}^2) \nabla \Tilde{u} , \nabla \Tilde{u} )_2 = \| \nabla \Tilde{u} \|_2^2 - 3 \int_{\Omega} \Tilde{u}^2 \abs{\nabla \Tilde{u}}^2 \, dx.
$$
This leads to the inequality
$$
\frac{1}{2}\dt\| \Tilde{u}(t)\|_2^2 + \| \Delta \Tilde{u}(t) \|_2^2 + \| \Tilde{u}(t) \|_2^2  + 3 \int_{\Omega} \Tilde{u}^2 \abs{\nabla \Tilde{u}}^2 \, dx \leq   \| \nabla \Tilde{u} \|_2^2.
$$
Using integration by parts, Cauchy-Schwarz inequality, and Young inequality on the right-hand side, we get
$$
\| \nabla \Tilde{u }\|_2^2 \leq - \int_{\Omega} \Tilde{u}  \, \Delta  \Tilde{u} \, dx \leq \|\Tilde{ u} \|_2 \| \Delta \Tilde{u} \|_2 \leq \frac{1}{2} \|\Tilde{u}\|_2^2 + \frac{1}{2}\| \Delta \Tilde{u}\|_2^2 .
$$
Substituting this back and absorbing terms gives
$$
\frac{1}{2}\dt\| \Tilde{u}(t)\|_2^2 + \frac{1}{2}\| \Delta \Tilde{u}(t) \|_2^2 + \frac{1}{2} \| \Tilde{u}(t) \|_2^2  + 3\int_{\Omega} \Tilde{u}^2 \abs{\nabla \Tilde{u}}^2 \, dx \leq  0.
$$
Integrating this differential inequality in time from $0$ to $\infty$, we conclude that $\|\Tilde{u}(t) \|_2^2 \in \rL^1(0,\infty).$ Denoting by $v$ the approximate solution of \eqref{eq:model data}, guaranteed by \autoref{prop: global DA}, we apply \autoref{thm: V norm} and we obtain the following result.
\begin{cor}[Data Assimilation for \eqref{eq: Cahn-Hilliard}] \mbox{} \\
Let $u_0$, $v_0 \in \rL^2(\Omega)$. Then, there exist $\mu_0$, $\delta_0>0$ such that for all $\mu>\mu_0$ and $0<\delta < \delta_0$ 
\begin{equation*}
    \| (\Tilde{u} -v)(t) \|_{\rL^2(\Omega)}  \to 0  \ \text{ exponentially, as } \ t \to \infty.
\end{equation*}
\end{cor}
{\bf Acknowledgements. }{\small Gianmarco Del Sarto, Matthias Hieber and Tarek Z\"{o}chling acknowledge the support from the DFG project FOR~5528. The research of Filippo Palma is carried on under the auspices of GNFM-INdAM.}

\end{document}